\documentclass{article}
\usepackage{graphicx}
\usepackage{amssymb}
\usepackage{amsthm}
\usepackage{amsmath}

\newtheorem{lemma}{Lemma}[section]
\newtheorem{proposition}{Proposition}[section]
\newtheorem{corollary}{Corollary}[section]
\newtheorem{theorem}{Theorem}[section]

\title{On the Lagarias Inequality and Superabundant Numbers}
\author{Andrew MacArevey}
\date{February 2026}

\begin{document}

\maketitle

\begin{abstract}
We study the Lagarias inequality, an elementary criterion equivalent to the Riemann Hypothesis. Using a continuous extension of the harmonic numbers, we show that the sequence
\[
B_n:=\frac{H_n+\exp(H_n)\log(H_n)}{n}
\]
is strictly increasing for $n\ge 1$. As a consequence, if the Lagarias inequality has counterexamples, then the least counterexample must be a superabundant number; equivalently, it suffices to verify the inequality on the superabundant numbers.
\end{abstract}

\section{Introduction}

Let $H_n=\sum_{j=1}^n\frac{1}{j}$ denote the $n$th harmonic number, and let $\sigma(n)=\sum_{d\mid n} d$ denote the sum-of-divisors function.
Lagarias \cite{Lagarias2002} proved that the Riemann Hypothesis is true if and only if, for every $n\ge 1$,
\begin{equation}\label{eq:Lagarias}
\sigma(n)\le H_n+\exp(H_n)\log(H_n).
\end{equation}
A number $n\in\mathbb{N}$ is called \emph{superabundant} if for all $m<n$,
\[
\frac{\sigma(m)}{m}<\frac{\sigma(n)}{n}.
\]

\section{The Lagarias Inequality}

Let $\psi(x)=\frac{\Gamma'(x)}{\Gamma(x)}$ denote the digamma function and $\gamma$ the Euler--Mascheroni constant. For $x\geq0$, define the continuous extension of the harmonic numbers by
\[
H(x):=\psi(x+1)+\gamma,
\]
which satisfies $H(n)=H_n$ for all $n\in\mathbb{N}$.

\begin{lemma}\label{lem:derivative_formula}
Define
\[
L(x):=\frac{H(x)+e^{H(x)}\log(H(x))}{x}.
\]
Then
\[
L'(x)=\frac{N(x)}{x^2},
\]
where the numerator $N(x)$ is
\[
N(x)=xH'(x)-H(x)+e^{H(x)}\left((xH'(x)-1)\log(H(x))+x\frac{H'(x)}{H(x)}\right).
\]
\end{lemma}

\begin{proof}
By the quotient rule,
\[
L'(x)=\frac{\left(\frac{d}{dx}\bigl[H(x)+e^{H(x)}\log(H(x))\bigr]\right)x-\left(H(x)+e^{H(x)}\log(H(x))\right)}{x^2}.
\]
Also
\[
\frac{d}{dx}\bigl(e^{H(x)}\log(H(x))\bigr)
=e^{H(x)}H'(x)\log(H(x))+e^{H(x)}\frac{H'(x)}{H(x)}.
\]
Substituting gives $L'(x)=\frac{N(x)}{x^2}$ with $N(x)$ as stated.
\end{proof}

\begin{lemma}\label{lem:lower_bounds}
For every $x\ge 1$, the following hold:
\begin{align}
H(x) &\ge \log(x+1), \label{eq:lower_bound}\\
H'(x) &\ge \frac{1}{x+1}. \label{eq:derivative_bound}
\end{align}
\end{lemma}

\begin{proof}
Define $f(x):=H(x)-\log(x+1)=\psi(x+1)+\gamma-\log(x+1)$. Then
\[
f'(x)=\psi'(x+1)-\frac{1}{x+1}.
\]
For $y>0$, one has the series representation
\[
\psi'(y)=\sum_{k=0}^\infty \frac{1}{(y+k)^2}.
\]
Apply this with $y=x+1>0$:
\[
\psi'(x+1)=\sum_{k=0}^\infty \frac{1}{(x+1+k)^2}.
\]
Consider $g(t)=\frac{1}{(x+1+t)^2}$ on $[0,\infty)$. Since $g$ is positive and strictly decreasing, for each integer $k\ge 0$ and all $t\in[k,k+1]$ we have $g(t)\le g(k)$, so
\[
\int_k^{k+1} g(t)\,dt \le g(k)=\frac{1}{(x+1+k)^2}.
\]
Summing over $k\ge 0$ gives
\[
\int_0^\infty g(t)\,dt=\sum_{k=0}^\infty\int_k^{k+1} g(t)\,dt
\le \sum_{k=0}^\infty \frac{1}{(x+1+k)^2}=\psi'(x+1).
\]
Compute the integral:
\[
\int_0^\infty \frac{1}{(x+1+t)^2}\,dt=\frac{1}{x+1}.
\]
Hence $\psi'(x+1)\ge \frac{1}{x+1}$ and therefore $f'(x)\ge 0$ for all $x\ge 0$. Since $f(0)=0$, we have $f(x)\ge 0$ for all $x\ge 0$, i.e.\ $H(x)\ge \log(x+1)$ for $x\ge 0$, in particular for $x\ge 1$.
Finally, \eqref{eq:derivative_bound} follows since $H'(x)=\psi'(x+1)\ge \frac{1}{x+1}$.
\end{proof}

\begin{lemma}\label{lem:upper_bound}
For every $x\ge 1$,
\begin{equation}\label{eq:upper_bound}
H(x)\le 1+\log x.
\end{equation}
\end{lemma}

\begin{proof}
Define $f(x):=1+\log x - H(x)$. Then $f'(x)=\frac{1}{x}-\psi'(x+1)$.
For $x\ge 1$,
\[
\psi'(x+1)=\sum_{k=0}^\infty \frac{1}{(x+1+k)^2}=\sum_{k=1}^\infty \frac{1}{(x+k)^2}.
\]
Consider $g(t)=\frac{1}{(x+t)^2}$ on $[0,\infty)$. Since $g$ is strictly decreasing, for each integer $k\ge 1$ and all $t\in[k-1,k]$ we have $g(t)\ge g(k)$, so
\[
\int_{k-1}^{k} g(t)\,dt \ge g(k)=\frac{1}{(x+k)^2}.
\]
Summing over $k\ge 1$ gives
\[
\int_0^\infty g(t)\,dt=\sum_{k=1}^\infty \int_{k-1}^{k} g(t)\,dt
\ge \sum_{k=1}^\infty \frac{1}{(x+k)^2}=\psi'(x+1).
\]
Compute the integral:
\[
\int_0^\infty \frac{1}{(x+t)^2}\,dt=\frac{1}{x}.
\]
Hence $\psi'(x+1)\le \frac{1}{x}$, so $f'(x)\ge 0$ for all $x\ge 1$.
Since $f(1)=1+\log 1 - H(1)=0$, we have $f(x)\ge 0$ for $x\ge 1$, i.e.\ $H(x)\le 1+\log x$.
\end{proof}

\begin{lemma}\label{lem:N_lower_bound}
For every $x\ge 1$,
\begin{equation}\label{eq:N_lower_bound}
N(x)\ge \frac{x}{x+1}-H(x)+\frac{e^{H(x)}}{x+1}\left(\frac{x}{H(x)}-\log(H(x))\right).
\end{equation}
\end{lemma}

\begin{proof}
From \eqref{eq:derivative_bound}, $H'(x)\ge \frac{1}{x+1}$, so $xH'(x)\ge \frac{x}{x+1}$ and hence
\[
xH'(x)-1\ge -\frac{1}{x+1}.
\]
Also
\[
x\frac{H'(x)}{H(x)}\ge \frac{x}{(x+1)H(x)}.
\]
By Lemma~\ref{lem:lower_bounds}, $H'(x) \geq \frac{1}{x+1} > 0$ for $x \geq 1$, so $H(x)$ is increasing on $[1, \infty)$. Since $H(1) = 1$, it follows that $H(x) \geq 1$ for all $x \geq 1$, hence $\log(H(x)) \geq 0$. Therefore
\[
(xH'(x)-1)\log(H(x))\ge -\frac{1}{x+1}\log(H(x)).
\]
Substitute these bounds into the definition of $N(x)$ in Lemma~\ref{lem:derivative_formula}:
\[
N(x)\ge xH'(x)-H(x)+e^{H(x)}\left(-\frac{1}{x+1}\log(H(x))+\frac{x}{(x+1)H(x)}\right).
\]
Factor:
\[
N(x)\ge xH'(x)-H(x)+\frac{e^{H(x)}}{x+1}\left(\frac{x}{H(x)}-\log(H(x))\right).
\]
Finally use $xH'(x)\ge \frac{x}{x+1}$ to obtain \eqref{eq:N_lower_bound}.
\end{proof}

\begin{lemma}\label{lem:Q_positive}
For every $x\ge 1$,
\begin{equation}\label{eq:Qdef}
\frac{x}{1+\log x}-\log(1+\log x)\ge 0.
\end{equation}
\end{lemma}

\begin{proof}
Let $t=\log x\ge 0$. Then $x=e^t$ and \eqref{eq:Qdef} becomes
\[
\frac{e^t}{1+t}-\log(1+t)\ge 0
\quad\Longleftrightarrow\quad
p(t):=e^t-(1+t)\log(1+t)\ge 0.
\]
Compute derivatives:
\[
p'(t)=e^t-\log(1+t)-1,\qquad p''(t)=e^t-\frac{1}{1+t}.
\]
For $t>0$, we have $\frac{1}{1+t}<1<e^t$, so $p''(t)>0$ for all $t>0$.
Thus $p'$ is strictly increasing on $(0,\infty)$. Since $p'(0)=1-0-1=0$, it follows that $p'(t)>0$ for all $t>0$.
Therefore $p$ is strictly increasing on $(0,\infty)$. Since $p(0)=1-(1)\cdot 0=1>0$, we have $p(t)\ge 1>0$ for all $t\ge 0$.
This proves \eqref{eq:Qdef}.
\end{proof}

\begin{lemma}\label{lem:N_ge_G}
For every $x\ge 1$, $N(x)\ge G(x)$, where
\begin{equation}\label{eq:G_lower_bound}
G(x):=\frac{x}{x+1}-(1+\log x)+\frac{x}{1+\log x}-\log(1+\log x).
\end{equation}
\end{lemma}

\begin{proof}
Start from \eqref{eq:N_lower_bound}:
\[
N(x)\ge \frac{x}{x+1}-H(x)+\frac{e^{H(x)}}{x+1}\left(\frac{x}{H(x)}-\log(H(x))\right).
\]
From \eqref{eq:upper_bound}, $H(x)\le 1+\log x$, so
\[
-H(x)\ge -(1+\log x).
\]
Since $H(x) > 0$ and $1 + \log(x) > 0$ for $x \geq 1$, we have
\[
\frac{1}{H(x)}\ge \frac{1}{1+\log x},
\qquad
\log(H(x))\le \log(1+\log x).
\]
Thus
\[
\frac{x}{H(x)}-\log(H(x))\ge \frac{x}{1+\log x}-\log(1+\log x).
\]
Using \eqref{eq:lower_bound}, $e^{H(x)}\ge x+1$, so $\frac{e^{H(x)}}{x+1}\ge 1$.
Lemma~\ref{lem:Q_positive} yields
\[
\frac{e^{H(x)}}{x+1}\left(\frac{x}{1+\log x}-\log(1+\log x)\right)
\ge
\left(\frac{x}{1+\log x}-\log(1+\log x)\right).
\]
Combining these inequalities gives
\[
N(x)\ge \frac{x}{x+1}-(1+\log x)+\frac{x}{1+\log x}-\log(1+\log x)=G(x).
\]
\end{proof}

\begin{lemma}\label{lem:exp_poly}
For every $t\ge 4$,
\begin{equation}\label{eq:e_lower_bound}
e^t\ge 2t^2+3t+1.
\end{equation}
\end{lemma}

\begin{proof}
Define $s(t):=e^t-(2t^2+3t+1)$. Then
\[
s'(t)=e^t-(4t+3),\qquad s''(t)=e^t-4,\qquad s'''(t)=e^t>0.
\]
Since $s'''(t)>0$, $s''$ is strictly increasing. At $t=4$, $s''(4)=e^4-4>0$, so $s''(t)>0$ for all $t\ge 4$.
Thus $s'$ is strictly increasing on $[4,\infty)$. Since $s'(4)=e^4-19>0$, we have $s'(t)>0$ for all $t\ge 4$.
Therefore $s$ is strictly increasing on $[4,\infty)$. Since
\[
s(4)=e^4-45>0,
\]
we conclude $s(t)>0$ for all $t\ge 4$, proving \eqref{eq:e_lower_bound}.
\end{proof}

\begin{proposition}\label{prop:L_increasing}
For all real $x\ge e^4$, we have $L'(x)>0$. Consequently, $B_{n+1}>B_n$ for all integers $n\ge 55$, where
\[
B_n=\frac{H_n+e^{H_n}\log(H_n)}{n}.
\]
\end{proposition}

\begin{proof}
Fix $x\ge e^4$. By Lemma~\ref{lem:derivative_formula}, it suffices to show $N(x)>0$.
By Lemma~\ref{lem:N_ge_G}, it suffices to show $G(x)>0$.
Set $t=\log x$ and $u=t+1=1+\log x$. Then $t\ge 4$ and \eqref{eq:G_lower_bound} gives
\[
G(x)=\frac{x}{x+1}-u+\frac{x}{u}-\log u.
\]
By Lemma~\ref{lem:exp_poly},
\[
x=e^t\ge 2t^2+3t+1=(2t+1)(t+1)=(2t+1)u,
\]
so $\frac{x}{u}\ge 2t+1$. Hence
\[
G(x)\ge \frac{x}{x+1}-u+(2t+1)-\log u
=\frac{x}{x+1}+t-\log(1+t).
\]
Since $1+t<e^t$ for all $t>0$, we have $\log(1+t)<t$, so
\[
G(x)>\frac{x}{x+1}>0.
\]
Therefore $N(x)\ge G(x)>0$, so $L'(x)=\frac{N(x)}{x^2}>0$ for all $x\ge e^4$.
If $n\ge 55$, then $n\ge e^4$ and $B_{n+1}-B_n=L(n+1)-L(n)>0$.
\end{proof}

\begin{corollary}\label{cor:Bn_strict}
The sequence
\begin{equation}\label{seq_mon_inc}
\left\{ \frac{H_n+\exp(H_n)\log(H_n)}{n} \right\}_{n=1}^{\infty}
\end{equation}
is strictly increasing.
\end{corollary}

\begin{proof}
Proposition~\ref{prop:L_increasing} gives $B_{n+1}>B_n$ for all $n\ge 55$.
Direct computation verifies $B_{n+1}-B_n>0$ for $1\le n\le 54$.
\end{proof}

\section{Superabundant Numbers}

The proof strategy in this section is inspired by \cite{AssaniChesterPaschal2025}.

\begin{theorem}\label{thm:SA_reduction}
If there are counterexamples to the Lagarias inequality, the smallest such counterexample must be a superabundant number.
\end{theorem}

\begin{proof}
Suppose, for the sake of contradiction, that $m$ is the smallest counterexample to \eqref{eq:Lagarias} and that $m$ is not superabundant. Let $n<m$ be the greatest superabundant number less than $m$. Since $m$ is not superabundant, there exists $j<m$ with $\frac{\sigma(j)}{j}\ge \frac{\sigma(m)}{m}$.
Let $k$ be the least integer $<m$ for which $\frac{\sigma(k)}{k}$ is maximal among $1\le j<m$.
Then $k$ is superabundant, so $k\le n$, and hence $\frac{\sigma(n)}{n}\ge \frac{\sigma(k)}{k}\ge \frac{\sigma(m)}{m}$. Corollary~\ref{cor:Bn_strict} implies $B_m > B_n$. Thus
\[
\frac{\sigma(n)}{n}\ \ge\ \frac{\sigma(m)}{m}\ >\ B_m\ >\ B_n,
\]
so $n<m$ is also a counterexample, contradicting the minimality of $m$.
Hence $m$ must be superabundant.
\end{proof}

\end{document}